\documentclass{article}

\usepackage{hyperref}
\usepackage{amsmath}
\usepackage{amsthm}
\usepackage{amssymb}
\usepackage{cite}
\usepackage{amsmath,amsthm}
\usepackage{amsfonts}
\usepackage{amssymb}
\usepackage{mathtools}
\usepackage{tikz}
\usepackage{enumitem}
\usepackage{lipsum}
\newcommand\blfootnote[1]{%
  \begingroup
  \renewcommand\thefootnote{}\footnote{#1}%
  \addtocounter{footnote}{-1}%
  \endgroup
}
\DeclareMathOperator{\im}{Im}
\usepackage{dirtytalk} 
\usepackage{import}
\newtheorem{theorem}{Theorem}[section]

\newtheorem{definition}[theorem]{Definition}
\newtheorem{example}[theorem]{Example}

\newtheorem{corollary}[theorem]{Corollary}
\newtheorem{proposition}[theorem]{Proposition}
\newtheorem{lemma}[theorem]{Lemma}
\newtheorem{problem}[theorem]{Problem}

\usepackage{lineno}

\begin{document}

\title{\bf On Infinitely Many Siblings for Locally Finite Trees with Parabolic Self-Embeddings
\blfootnote{This is a project as part of author's thesis under supervision of Dr. Claude Laflamme and Dr. Robert Woodrow at the Department of Mathematics and Statistics, University of Calgary, Calgary, AB, Canada (2017-2022).}}

\author{Davoud Abdi} 

\maketitle              

\begin{abstract}
Parabolic (resp. hyperbolic) self-embeddings of trees are those which do not fix a non-empty finite subtree and preserve precisely one (resp. two) end(s). 
We prove that a locally finite tree having a parabolic self-embedding is mutually embeddable with infinitely many pairwise non-isomorphic trees, unless the tree is a one-way infinite path. As a result, we conclude that two important properties identified by  Bonato-Tardif and Tyomkyn hold for locally finite trees not having any hyperbolic self-embedding. 
\end{abstract}

\section{Introduction}

Trees in this context are in graph theoretical sense, that is, connected and acyclic simple graphs.
An {\em embedding} from a tree $T$ to another tree $S$ is an injective map from the vertex set of $T$ to the vertex set of $S$ preserving the adjacency relation. A tree $T$ {\em embeds} in another tree $S$, denoted by $T\hookrightarrow S$, if there is an embedding from $T$ to $S$. Two trees $T, S$ are  {\em equimorphic} or {\em siblings}, denoted by $T\approx S$, if there are mutual embeddings between them. An embedding from $T$ to itself is called a {\em self-embedding} of $T$. The set of all self-embeddings of a tree $T$ forms a monoid under composition of functions called the {\em monoid of self-embeddings} of $T$, denoted by $Emb(T)$.  Clearly, two finite equimorphic trees are isomorphic, however, this is no longer the case for infinite trees. For instance, a tree consisting of a vertex $r$ and countably many paths of length 2 attached to $r$ has countably many siblings, up to isomorphism. The number of isomorphism classes of siblings of a tree $T$ is called the {\em sibling number} of $T$, denoted by $Sib(T)$.

\begin{definition} [The Tree Alternative Property, TAP in short, \cite{LPS}] 
A tree T has the Tree Alternative Property if $Sib(T)=1$ or $\infty$.
\end{definition}

Bonato and Tardif \cite{BT} (2006) conjectured that TAP holds for each tree ({\em The Tree Alternative Conjecture}). As the first step, they \cite{BT} proved that rayless trees satisfy TAP. Their proof relies on a fixed point theorem of Halin \cite{HA}, which states that every rayless tree $T$ has either a vertex or an edge that is fixed by every self-embedding of $T$. TAP was  verified for rooted trees by Tyomkyn \cite{TY} (2009), who also made some progress towards TAP for locally finite trees. 
 
\begin{definition} [The Abundant Sibling Property, ASP in short]
A tree T has the Abundant Sibling Property if  $Sib(T)=\infty$.
\end{definition} 

Tyomkyn \cite{TY} conjectured that if a locally finite tree has a non-surjective self-embedding, then it has ASP, unless the tree is a one-way infinite path ({\em Tyomkyn's Conjecture}). 
Laflamme, Pouzet, Sauer \cite{LPS} (2017) used Halin's fixed point theorem to prove  TAP for scattered trees, that is trees not containing a subdivision of the complete binary tree. Indeed, they verified  TAP for a more general class of trees namely stable trees. They also showed that ASP holds for locally finite scattered trees. Hamann \cite{HAM} (2019) used a result by Laflamme, Pouzet and Sauer to prove that a tree either is non-scattered or has a vertex, an edge, an end or two ends fixed by all its self-embeddings. With the help of his analysis of the monoid of self-embeddings of trees, Hamann \cite{HAM} also gave a precise description of the open cases for  TAP and proved that  TAP holds for a tree whose monoid of self-embeddings does not satisfy two structural properties. 

While working to extend Hamann’s results, the author was informed by Tyomkyn (personal communication, \cite{TYP}) of an unpublished manuscript by Tateno \cite{TAT} (2008) claiming the construction of a locally finite tree with an arbitrary finite number of siblings. Together with Laflamme, Tateno and Woodrow, we undertook to revisit the ideas carefully and provide the details, (see \cite{ALTW}). Thus, there are locally finite trees having an arbitrary finite number of siblings. Moreover, these trees have non-surjective self-embeddings and hence have neither TAP nor ASP. 
This is a major development in the programme of understanding siblings of a given tree. While the first approach was toward proving all trees have TAP, the equimorphy programme has moved on and focused on the actual structure of those siblings. In other words, it is interesting to establish the boundaries and determine the structure of those trees satisfying TAP or ASP.

In this article we report on work from before learning of Tateno’s claim investigating TAP using the monoid of self-embedding. 
Our method is based on the notion of {\em end} which was first defined by Halin \cite{HAL} in terms of equivalence classes of infinite paths: two infinite paths in a tree are {\em equivalent} if their intersection is also an infinite path. Halin \cite{HAL} also classified the self-embeddings of trees according to type 1 and type 2 which are respectively called elliptic and non-elliptic self-embeddings by Hamann \cite{HAM}.  We first show in Section \ref{Stabletheorem} that TAP holds in one of the two open cases of Hamann mentioned earlier. Then, in Section \ref{LFTrees}, we will study locally finite trees by analysing their self-embeddings. 
Elliptic self-embeddings of trees are those which fix a finite subtree. Non-elliptic self-embeddings of trees are of two types: parabolic and hyperbolic. Parabolic (resp. hyperbolic) self-embeddings are those which do not fix a non-empty finite subtree and fix precisely one (resp. two) end(s). Halin \cite{HAL} showed that elliptic self-embeddings of locally finite connected graphs are automorphism. We will show that if a locally finite tree $T$ has a parabolic self-embedding, then $T$ has infinitely many siblings, unless $T$ is a one-way infinite path. These two results together imply that both TAP and ASP hold for a locally finite tree not having any hyperbolic self-embedding. This poses a restriction on locally finite trees making them have TAP. Given the locally finite tree examples constructed in \cite{ALTW} which satisfy neither TAP nor ASP, this article ends with asking what further structural conditions may be added to ensure a locally finite tree satisfies TAP or ASP.

\section{A Fixed Point Theorem} \label{Stabletheorem}

A {\em ray} (resp. {\em double ray}) is a one-way (resp. two-way) infinite path. 
Let $T$ be a tree. Two rays $R_1$ and $R_2$ in $T$ are  {\em equivalent}, denoted by $R_1 \sim R_2$, if their intersection is also a ray. This is an equivalence relation whose classes are the {\it ends} of $T$. 
Therefore, an end 
$\eta$ of $T$ consists of equivalent rays i.e. $\eta=[R]_\sim$ for some ray $R$ in $T$. Let $\eta$ be an end of $T$. By an $\eta$-ray $R$, we mean $R\in \eta$. The set of ends of $T$ is denoted by $\Omega(T)$.  
We say that a vertex $r\in T$ \textit{separates} a vertex $v\in T$ from an end $\eta\in \Omega(T)$, when the unique ray in $\eta$ starting at $v$ passes through $r$.  A vertex $r\in T$ separates two ends $\eta, \mu\in \Omega(T)$ when the unique double ray between $\eta$ and $\mu$ passes through $r$. A sequence of vertices $(x_n)_{n<\omega}$ of $T$ \textit{converges} to an end $\eta\in \Omega(T)$, denoted by $x_n\to\eta$, if there exists an $\eta$-ray $R=r_0r_1\cdots$ such that for every $n<\omega$, $r_n$ separates only finitely many vertices of the sequence from $\eta$. We also say that a sequence of ends $(\mu_n)_{n<\omega}$ converges to an end $\eta$, denoted by $\mu_n\to\eta$, if there exists an $\eta$-ray $R=r_0r_1\cdots$ such that for every $n$, $r_n$ separates only finitely many elements in the sequence $(\mu_n)_{n<\omega}$ from the end $\eta$. 
For instance, 
let $T$ consist of a ray $R=r_0r_1\cdots$ and new vertices $x_n$, $n < \omega$, such that each $x_n$ is attached to $r_n$ by an edge. The sequence $(x_n)_{n < \omega}$ converges to the end containing the ray $R$ because for every $n<\omega$, $r_n$ separates only the vertices $x_m$ with $m\leq n$. 
Let $(x_n)_{n<\omega}$ be a sequence of vertices of $T$ such that $x_n\to\eta$ for some $\eta\in \Omega(T)$. Note that whenever $(y_n)_{n<\omega}$ is another sequence of vertices so that for some positive integer $N$, $d(x_n,y_n)\leq N$ for every $n<\omega$, then $y_n\to\eta$. This property is called \textit{projectivity} and it is true in every tree (\cite{HAM}). Let $M$ be a submonoid of $Emb(T)$. The \textit{limit set} $\mathcal{L}(M)$ is the set of accumulation points of $\{g(x): g\in M\}$ in $\Omega(T)$ for any $x\in T$. By projectivity, the limit set is independent from the choice of $x$. 
A self-embedding $f$ of $T$ preserves a ray $R$ in $T$   \textit{forward}, (resp. \textit{backward}), if $f(R)\subseteq R$, (resp. $R\subseteq f(R)$), and $f$ preserves an end $\eta$ of $T$  forward, (resp. backward), if $f$ preserves some $\eta$-ray forward, (resp. backward). We also say that a submonoid $M$ of $Emb(T)$ preserves an end $\eta$ forward, (resp. backward), if every $f\in M$ preserves $\eta$ forward, (resp. backward) \cite{HAM}.  We follow Hamann \cite{HAM} to name the elements of $Emb(T)$. 
A self-embedding $g$ of $T$ is  
\begin{itemize}
    \item {\it elliptic}, if $g$ fixes a non-empty finite subtree of $T$,
    \item {\it parabolic}, if $g$ is not elliptic and it fixes precisely one end of $T$, and
    \item {\it hyperbolic}, if $g$ is not elliptic and it fixes precisely two ends of $T$. 
\end{itemize}
Halin \cite{HAL} proved that when $g\in Emb(T)$ is elliptic, then $g$ fixes a vertex or an edge, and if it is non-elliptic, then there exists a ray $R$ in $T$ such that $g(R)\subset R$. This ray can be extended to either a maximal ray $R^\prime$ with $g(R^\prime)\subset R^\prime$ or a $g$-invariant double ray. In the first, (resp. second), case $g$ is parabolic (resp. hyperbolic). Indeed, the set $\{x:d(x,g(x))\ \text{is minimal}\}$ is the fixed ray, double ray, edge or the set of fixed vertices. Thus, every self-embedding of $T$ is either elliptic, parabolic or hyperbolic. 

Let $g\in Emb(T)$ be non-elliptic, where $T$ is a tree. Let $\eta_1$ and $\eta_2$ be two ends of $T$ preserved forward by $g$. If $\eta_1\neq \eta_2$, then consider the unique double ray $Z$ between $\eta_1$ and $\eta_2$. It follows that $Z$ has two tails $R_1\in\eta_1$ and $R_2\in\eta_2$ with $g(R_1)\subset R_1$ and 
$g(R_2)\subset R_2$, which is impossible. The unique end fixed by $g$ and determined by the sequences $x, g(x), g^2(x), \ldots$, $x\in T$, will be called the \textit{direction} of $g$ and denoted by $g^+$. When $g$ is hyperbolic, we denote by $g^-$ the unique $g$-invariant end other than $g^+$. When $\eta\in \Omega(T)$ is a direction, it means that there exists a non-elliptic self-embedding $g$ of $T$ such that $g^+=\eta$. Let $M$ be a submonoid of $Emb(T)$. The set of all directions of non-elliptic elements in $M$ is denoted by $\mathcal{D}(M)$  (\cite{HAM}). 

It is clear that in finite trees both the limit set and the set of directions are empty because there is no end. Let $\mathcal{R}$, (resp. $\mathcal{Z}$), be the tree consisting of a ray, double ray. Then, by setting $M:=Emb(\mathcal{R})$ and $N:=Emb(\mathcal{Z})$, we have $|\mathcal{L}(M)|=|\mathcal{D}(M)|=1$ and $|\mathcal{L}(N)|=|\mathcal{D}(N)|=2$.
The following example shows that in a tree, there might be only one limit point but no direction. This example is from a personal communication with Hamann. 

\begin{example}[\cite{HAMP}, $\mathcal{D}(M)=\emptyset$ but $|\mathcal{L}(M)|=1$] 
Let $\mathcal{T}^3$ be the 3-regular tree, that is the countable tree for which each vertex has degree precisely 3. 
Let $\eta$ be an end of $\mathcal{T}^3$. Let $M$ be the submonoid of  $Emb(\mathcal{T}^3)$ such that each element of $M$ fixes all vertices of some ray in $\eta$. To see that this is a monoid, take the composition of two elements $f, g$ of $M$ and let $R_f, R_g$ be the rays in $\eta$ that are fixed by $f, g$ pointwise, respectively. Then $R_f\cap R_g$ is an $\eta$-ray which is fixed by $f\circ g$ pointwise. Therefore, $f\circ g$ lies in $M$. In particular, all elements of $M$ are elliptic and hence $\mathcal{D}(M)$ is empty. 

Now we show that $\eta$ is the unique limit point of M. 
Let $R=x_0x_1, \ldots$ be an $\eta$-ray and for each $n$ let $g_n$ be an self-embedding in M that fixes precisely the vertices $x_n, x_{n+1}, \ldots$. Then, the sequence $x_0, g_0(x_0), g_1(x_0), \ldots$ converges to $\eta$. Therefore, $\eta$ is a limit point of M.
Suppose that there is a second limit point $\mu\neq\eta$ of M. Since $\mu$ is a limit point, there is a sequence $f_0, f_1, f_2, \ldots$ of self-embeddings in M such that $x_0, f_0(x_0), f_1(x_0), \ldots$ converges to $\mu$. Let $S=y_0, y_1, \ldots$ be a ray in $\mu$ such that $Z=\ldots, y_2, y_1, y_0, x_n, x_{n+1}, ...$ is the double ray between $\mu$ and $\eta$. Since $f_n\in M$ and all $\eta$-rays are equivalent, each $f_n$ fixes a tail of $R$ pointwise. Therefore, we conclude that there is some $k\geq n$ such that $d(x_0,x_k)=d(f_n(x_0),x_k)$. As a consequence, $y_k$ separates $\mu$ from all $f_n(x_0)$, where $n$ is such that $d(x_0,x_n)+1=k$. Thus, the sequence $x_0, f_0(x_0), f_1(x_0), \ldots$ cannot converge to $\mu$, a contradiction. Thus, $\eta$ is the only limit point of M. 
\end{example}

Let $\Gamma, \Delta\subseteq \Omega(T)$. We say that $\Gamma$ is {\em dense} in $\Delta$ if for $\eta\in \Delta$ there is a sequence $(\mu_n)_{n<\omega}$ of elements of $\Gamma$ converging to $\eta$.  

\begin{theorem} [\cite{HAM}] \label{SizeofDm}
Let $M$ be a submonoid of $Emb(T)$.
\begin{enumerate}
\item If $|\mathcal{L}(M)|\geq 2$, then $\mathcal{D}(M)$ is dense in $\mathcal{L}(M)$.
\item The set $\mathcal{L}(M)$ has either zero, one, two or infinitely many elements.
\item The set $\mathcal{D}(M)$ has either zero, one, two or infinitely many elements.
\end{enumerate} 
\end{theorem}

The canonical theorem of Hamann \cite{HAM} given below is called \say{fixed point theorem for self-embeddings} and it determines all possible cases for the monoid of self-embeddings of a tree. 

\begin{theorem}[\cite{HAM}] \label{Classification}
Let T be a tree and $M$ be a submonoid of $Emb(T)$. Then one and only one of the following holds.
\begin{enumerate}[label=\roman*)]
\item $M$ fixes either a vertex or an edge of $T$;
\item $M$ fixes a unique element of $\mathcal{L}(M)$;
    \item $\mathcal{L}(M)$ consists of precisely two elements;
    \item $M$ contains two non-elliptic elements that do not fix the direction of the other.
\end{enumerate}
\end{theorem}

When case (iv) in  Theorem \ref{Classification} occurs, it can be shown that there is a submonoid of $M$ generated freely by two non-elliptic self-embeddings and furthermore, $T$ contains a subdivision of the complete binary tree as a subtree (see \cite{HAM}, Theorem 3.2). This fact amounts to a theorem due to Laflamme, Pouzet and Sauer asserting that if a tree $T$ is a scattered tree, then $Emb(T)$ fixes either a vertex, an edge, or a set of at most two ends of $T$ (see \cite{LPS}, Theorem 1.1). For instance, the 3-regular tree $\mathcal{T}^3$ (the tree where each vertex has degree 3) is non-scattered and there are two non-elliptic self-embeddings $f, g$ of $\mathcal{T}^3$ so that neither $f$ fixes $g^+$ nor $g$ fixes $f^+$.

Let $T$ be a tree and $R$ a ray in $T$. For every $r \in R$ we denote by $T^R_r$  the maximal subtree of $T$ containing $r$ and edge-disjoint from $R$. A ray $R$ is {\it regular} (resp. {\it non-regular}), if the number of pairwise non-equimorphic subtrees $T^R_r$ is finite (resp. infinite). An end $\eta \in \Omega(T)$ is regular (resp. non-regular) if it contains a regular (resp. non-regular) ray. Hamann \cite{HAM} provided the following criterion towards TAP. 

\begin{theorem} [\cite{HAM}]\label{Opencases}
TAP holds for all trees $T$ whose monoid $Emb(T)$ of self-embeddings \textbf{does not} satisfy the following properties.
\begin{enumerate}
    \item There is a regular end $\eta$ of $T$ such that $Emb(T)$ preserves $\eta$ forward.
    \item There is a submonoid of $Emb(T)$ generated freely by two non-elliptic elements. 
\end{enumerate}
\end{theorem}

An end $\eta$ of $T$ is {\it almost rigid} if every self-embedding of $T$ preserves $\eta$ forward and backward, that is every self-embedding of $T$ fixes pointwise a cofinite subset of every ray belonging to $\eta$.

\begin{theorem}[\cite{LPS}, Theorem 1.3] \label{Regularnonalmostrigid}
Let $T$ be a tree containing a regular end $\eta$ which is not almost rigid. If $Emb(T)$ preserves $\eta$ forward, then $Emb(T)$ preserves some $\eta$-ray. 
\end{theorem}

A tree $T$ is  {\it stable} if either there exists a non-regular end $\eta$ such that $Emb(T)$ preserves $\eta$ forward; or there is an almost rigid end; or $Emb(T)$ preserves a vertex or an edge or a ray or a double ray (\cite{LPS}).  Laflamme, Pouzet and Sauer proved that stable trees have TAP (see \cite{LPS}, Theorem 1.9). 
We show that a tree $T$ satisfying the first case of Theorem \ref{Opencases} is stable and consequently it has TAP.

\begin{theorem} \label{RegularStable}
Let T be a tree. If there is a regular end of $T$ preserved forward by all self-embeddings of T, then T is stable. Consequently, T has TAP. 
\end{theorem}

\begin{proof}
For simplicity, set $M:=Emb(T)$. Let $\eta$ be a regular end preserved forward by $M$.
First note that $T$ cannot have more than one direction because otherwise there exists some direction $\mu$ other than $\eta$, that is $\mu=g^+$ for some non-elliptic $g\in M$, which implies that $g$ does not preserve $\eta$ forward, a contradiction. More, $M$ satisfies one and only one of the cases (i)-(iv) of Theorem \ref{Classification}. 
Cases (iii) and (iv) amount to having more than one direction. Suppose that case (i) occurs and let $F=\{x,y\}$ be a vertex or an edge fixed by $M$. Let $R_x$ be the $\eta$-ray starting from $x$ and let $g\in M$ be given. Then, by assumption there is a ray $R_g\in \eta$ with $g(R_g)\subseteq R_g$. Set $R:=R_x\cap R_g$ which is an $\eta$-ray. Since $g$ fixes $F$, it follows that $g$ is identity on $R$ and consequently the ray $R_x$ is fixed by $g$. Hence, $R_x$ is fixed by $M$. Finally, assume that $M$ fixes a unique element of $\mathcal{L}(M)$ which is $\eta$.
First assume that there is no direction. Then all self-embeddings of $T$ are elliptic. Let $f\in M$. Since $f$ is elliptic and fixes $\eta$, it fixes a vertex and hence the ray in $\eta$ starting at this vertex. Indeed, all elements of $M$ preserve $\eta$ forward and backward meaning that $\eta$ is almost rigid. Consequently, $T$ is stable in this case.
Now suppose that there is only one direction. The unique direction must be $\eta$ which is regular. Since $\eta$ is a direction, $\eta=f^+$ for some non-elliptic $f\in M$. Since $f$ preserves $\eta$ forward and since $f$ is non-elliptic, there exists an $\eta$-ray $R$ so that $f(R)\subset R$. Clearly, $f$ does not fix pointwise a cofinite subset of $R$ meaning that $\eta$ is not almost rigid. By Theorem \ref{Regularnonalmostrigid}, it follows that $M$ preserves some $\eta$-ray. 

Therefore, in all cases, $T$ is stable and it has TAP.
\end{proof}

The proof of Theorem \ref{RegularStable} implies the following statement which will be used later to construct infinitely many siblings for a locally finite tree having a parabolic self-embedding. 

\begin{corollary} \label{Regularray} 
Let T be a tree and assume that there is a regular end of T preserved forward by all self-self-embeddings of T. If there is a vertex or an edge fixed by $Emb(T)$, or if there is a non-elliptic self-embedding of T, then there is a ray preserved by $Emb(T)$. 
\end{corollary}

\section{Siblings of Locally Finite Trees} \label{LFTrees}

In this section we prove that a locally finite tree not having any hyperbolic self-embedding has both TAP and ASP. Throughout this section $T$ represents a locally finite tree and by $M$ we mean the monoid of all self-embeddings of $T$. 
Our method will be based on analysing the size of the set $\mathcal{D}(M)$ of all directions of non-elliptic elements in $M$. By Theorem \ref{SizeofDm}, $\mathcal{D}(M)$ has either zero, one, two or infinitely many elements. We show that if the size of $\mathcal{D}(M)$ is not infinite, then $T$ has TAP. We also show that if a non-regular end of  $T$ is preserved forward or backward by some non-elliptic self-embedding, then $T$ has ASP (Corollary \ref{Nonregularend}). Further, we show that if $T$ has a parabolic self-embedding, then $T$ has ASP, unless $T$ is a ray (Theorem \ref{ParabolicInfinite}).

\subsection{Ingredients}   \label{Ing}

In this subsection, we use some results from Tyomkyn \cite{TY} and list all tools we need.

Let $g$ be a non-elliptic self-embedding of $T$. Then $g$ preserves some ray $R$ forward. If $g$ is parabolic (resp. hyperbolic) then the ray $R$ is extended to a unique maximal ray (resp. double ray), denoted by $R_g$ (resp. $Z_g$), and we have $g(R_g)\subset R_g$ (resp. $g(Z_g)=Z_g$).  For every $r\in R_g$ ($\in Z_g$), we denote by $T^g_r$ the maximal subtree of $T$ containing $r$ and edge-disjoint from $R_g$ ($Z_g$). 
A {\em rooted tree}, denoted by $(T,r)$, is a tree $T$ with a
distinguished vertex $r$, called the {\em root}, such that each self-embedding of $T$ fixes $r$ (\cite{TY}).
The next lemma implies that equimorphism and isomorphism are equivalent for  rooted locally finite trees.

\begin{lemma} [\cite{TY}] \label{MutualrootedIso}
Let T be a locally finite tree. For each $r\in T$, any self-embedding of $(T, r)$ into itself is surjective. 
\end{lemma}

Lemma \ref{MutualrootedIso} asserts that given a non-elliptic self-embedding $g\in M$, if $(T^g_r,r)\approx (T^g_s,s)$, then $(T^g_r,r)\cong (T^g_s,s)$, where $r, s \in R_g (\in Z_g)$. In other words, $Sib((T^g_r,r))=1$ for each non-elliptic $g\in M$ and every $r\in R_g (\in Z_g)$.  

\begin{theorem} [\cite{TY}] \label{Infinitecomponents}
If there exists a self-embedding $f$ of a locally finite tree $T$ such that $T\setminus f(T)$ has infinitely many components, then $Sib(T)=\infty$.  
\end{theorem}

One of the main tools is the notion of a {\em rake}. The concept of rake goes back to the notion of comb defined by Diestel \cite{D}. A {\em comb} is a ray $R$ with infinitely many disjoint finite paths having precisely their first vertex on $R$. Tyomkyn modified this concept by saying that a comb is a ray with infinitely many disjoint non-trivial paths of finite length attached to it \cite{TY}. Indeed, Tyomkyn used the modification of a comb in order to obtain infinitely many components in the proof of Theorem \ref{Infinitecomponents}. Our definition of a rake prevents it from having a vertex with infinitely many neighbours. 
We call a tree $S$ a \textit{rake} if $S$ consists of a ray $R$ with infinitely many vertices attached to distinct vertices of $R$. The vertices of $S\setminus R$ which have degree 1 are called \textit{teeth} and the ray $R$ is called the {\em spine} of $S$.  
A tree $T$ is \textit{nearly finite} when $T$ is a finite tree with finitely many rays attached to it \cite{TY}. Equivalently, $T$ is nearly finite if it is locally finite and has only finitely many vertices of degree $3$ or more. 
We provide a proof for the following lemma based on our definition of a rake. The lemma is given in \cite{TY} which was proved based on the notion of comb.  We shall remind K\"onig's lemma (Lemma 8.1.2, \cite{D}) asserting that any infinite locally finite graph contains a ray. 

\begin{lemma}[\cite{TY}] \label{rake}
A tree $T$ is nearly finite if and only if it is locally finite and does not contain a rake as a subtree.
\end{lemma}

\begin{proof}
By definition, a nearly finite tree is locally finite and since a rake contains infinitely many vertices of degree at least 3, a nearly finite tree cannot contain a rake establishing  the \say{only if } part. 

For the \say{if } part, suppose $T$ is not nearly finite. Then either it is not locally finite which contradicts the assumption, or it has infinitely many vertices of degree at least 3. Let $G$ be the minimal subtree of $T$ containing all vertices of degree 3 or more.  By K\"onig's lemma, $G$ contains a ray $R$. The ray $R$ contains infinitely many vertices of degree at least 3 in $T$ because otherwise $R$ intersects only finitely many of such vertices which contradicts the minimality of $G$. For every $r\in R$ of degree at least 3 in $T$, let $v_r$ be a vertex in $T$ not in $R$ and adjacent to $r$. Define $S$ to be the subtree of $T$ consisting of $R$ and the $v_r$. The subtree $S$ of $T$ is a rake meaning that $T$ contains a rake, a contradiction. 
\end{proof}

By Theorem \ref{Infinitecomponents} and Lemma \ref{rake} one can prove the following. 

\begin{theorem}[\cite{TY}] \label{Infinitesiblingrake}
Let T be a locally finite tree. If there exists a self-embedding $f$ of $T$ such that some component of $T\setminus f(T)$ is not nearly finite, then $Sib(T)=\infty$.
\end{theorem}

\subsection{No Direction and Two Directions} 

Let $T$ be a locally finite tree. 
We show that if there is no direction or there are only two directions, then $T$ has only one sibling. When there is no direction in $\Omega(T)$, then $T$ does not have non-elliptic self-embedding. 
Halin proved the following. 

\begin{proposition} [\cite{HAL}, Corollary 6] \label{HAL}
Each elliptic self-embedding of a locally finite connected graph is an automorphism. 
\end{proposition}

Proposition \ref{HAL} implies that a non-surjective self-embedding of a locally finite tree $T$ is non-elliptic. As a corollary of Proposition \ref{HAL} we get the following. 

\begin{proposition} \label{Ellipticauto}
Let  M be the monoid of self-embeddings of a locally finite tree T. If  $\mathcal{D}(M)=\emptyset$, then $Sib(T)=1$. 
\end{proposition} 

\begin{proof}
Let $T'\subseteq T$ be a sibling of $T$ and $f:T\to T'$ a self-embedding. Since $\mathcal{D}(M)=\emptyset$, $f$ is elliptic and by Proposition \ref{HAL}, $f$ is an automorphism. Thus, $T\cong T'$ meaning that $Sib(T)=1$. 
\end{proof}

The {\em periodicity} of a non-elliptic self-embedding $f$ of $T$, written $p(f)$, is the length of the path $rPf(r)$ for any $r$ in $R_f$ (resp. $Z_f$) when $f$ is parabolic (resp. hyperbolic).

\begin{proposition} \label{Twodirectionssibone}
Let T be a locally finite tree and M  the monoid of self-embeddings of T. If $|\mathcal{D}(M)|=2$, then $Sib(T)=1$. 
\end{proposition}

\begin{proof}
Since $|\mathcal{D}(M)|=2$, we get $|\mathcal{L}(M)|\geq 2$. Further, $\mathcal{D}(M)$ is dense in $\mathcal{L}(M)$ by Theorem \ref{SizeofDm} (1). Hence, $|\mathcal{L}(M)|=2$.  According to a result of Hamann (Proposition 2.8 (ii), \cite{HAM}), when $|\mathcal{L}(M)|=2$, then all non-elliptic self-embeddings of $T$ are hyperbolic. Therefore, it suffices to show that each hyperbolic self-embedding of $T$ is an automorphism because elliptic self-embeddings of locally finite trees are automorphisms by Proposition \ref{HAL}. 

Let $f$ be a hyperbolic self-embedding of $T$.   
Let $Z$ be the unique double ray preserved by $M$. Then $Z$ is the unique double ray between $f^-$ and $f^+$. If $\eta$ is a direction other than $f^+$, then $\eta=f^-$ because otherwise $|\mathcal{D}(M)|=\infty$, a contradiction. By definition, $f^-=g^+$ for some hyperbolic $g\in M$. An immediate conclusion is that $g^- = f^+$. Let $p(f)=m$ and $p(g)=n$ and $r \in Z$ be given. We have $d(r, f^n(r))=nm$ and $d(f^n(r), g^m(f^n(r)))=nm$. It follows that $g^m (f^n(r))=r$.   Thus,  $(T^f_r,r) \hookrightarrow (T^f_{f(r)},f(r)) \hookrightarrow (T^f_{f^n(r)},f^n(r)) \hookrightarrow (T^f_r,r)$ and by Lemma \ref{MutualrootedIso}, $(T^f_r,r) \cong (T^f_{f(r)},f(r))$. The embedding  $f_{\upharpoonright T^f_r}$ defines this isomorphism. In other words, $f_{\upharpoonright T^f_r}$ is an onto embedding from $T^f_r$ to $T^f_{f(r)}$. Hence, $f$ is invertible meaning that $f$ is an automorphism.  
\end{proof}

\subsection{Parabolic Self-Embeddings} \label{SiblingsofT}

Our aim in this section is to find infinitely many siblings of a locally finite tree $T$ when there is a parabolic self-embedding of $T$, unless $T$ is a ray. In case the direction of the parabolic self-embedding is non-regular, we use Theorem \ref{Infinitecomponents} and when the direction is regular, we construct a strictly decreasing sequence of subtrees of $T$ w.r.t inclusion and prove that they contain infinitely many pairwise non-isomorphic siblings of $T$. 

Let $f$ be a non-elliptic self-embedding of $T$.  For two distinct vertices $s, t\in R_f$ ($\in Z_f$), we say $s$ is on the \textit{left side} of $t$ with respect to $f$ (equivalently, $t$ is on the \textit{right side} of $s$ with respect to $f$), denoted by $s<_f t$, if the component of $R_f\setminus\{s\}$ (resp. $Z_f\setminus \{s\}$) containing $t$ belongs to $f^+$. For two vertices $s, t\in Z_f$, by $s\leq_f t$ we mean $s<_f t$ or $s=t$. Let $R$ be a ray contained in $Z_f$.
By $s\leq_f R$ (resp. $R\leq_f s$), we mean $s\leq_f r$ (resp. $r\leq_f s$) for every $r\in R$.

Let $f$ be a parabolic self-embedding of a locally finite tree $T$. We consider the case $f^+$ is regular and the case $f^+$ is non-regular separately. First we construct a decreasing sequence of siblings of $T$ w.r.t embeddability by means of $R_f$. 

\begin{lemma} \label{Parabolicregular}
Let $T$ be a locally finite tree which is not a ray and $f$ a parabolic self-embedding of T. If the direction $f^+$ is regular, then $Sib(T)=\infty$. 
\end{lemma} 

\begin{proof}
Let $R_f=r_1r_2\cdots$ be the maximal ray preserved by $f$.
If for some $r\in R_f$, $T^f_r$ contains a rake, then by choosing sufficiently large $k$, we have $T^f_r\notin \im(f^k)$ and consequently $T\setminus f^k(T)$ contains a rake. This implies that $Sib(T)=\infty$ by Theorem \ref{Infinitesiblingrake}. 

Assume that for no $r\in R_f$, $T^f_r$ contains a rake. In other words, each $T^f_r$, $r\in R_f$, is nearly finite. Since $T$ is not a ray, there is a vertex $r\in R_f$ with $deg(r)\geq 3$. Let $s\in R_f$ be the minimal vertex with respect to $\leq_f$ such that $deg(s)\geq 3$. Such a vertex exists because $R_f$ has the minimum vertex $r_1$ w.r.t $\leq_f$. Let $g$ be a self-embedding of $T$ satisfying $g^+\neq f^+$ and let $r\in R_f$ separate $f^+$ from $g^+$. Then $deg(g^n(s))\geq 3$ for each $n$. It follows that $T^f_r$ contains a rake,  a contradiction. Therefore, for each non-elliptic $g\in M$, $g^+=f^+$. This means that the regular end $f^+$ is preserved forward by $M$. Thus, by Corollary \ref{Regularray}, there is a ray $R$ preserved by $M$.  Clearly $R\subseteq R_f$ meaning that the starting vertex of $R$ is some $r_t$. We will construct a strictly decreasing sequence $(S_k)_{k<\omega}$ of siblings of $T$ w.r.t inclusion. 

\noindent{\bf\large Siblings $S_k$:}
For $k\geq 1$, let $S_k$ be the tree obtained from $T$  by replacing $T^f_{f(s)}, \ldots , T^f_{f^k(s)}$ in $T$ with trivial trees.  Since each $S_k$ is a subtree of $T$, it embeds into $T$. On the other hand, the mapping $f^{k+1}$ is a self-embedding from $T$ into $S_k$. This means that $T\approx S_k$. Clearly, every $S_{k+1}$ is  a proper subtree of $S_k$ and $S_1$ is  a proper subtree of $T$. Therefore, we have $T \supset S_1 \supset S_2 \supset \cdots$.    

\noindent 
We prove that infinitely many of the $S_k$ are pairwise non-isomorphic. First we show that infinitely many of them are non-isomorphic to $T$. For the sake of a contradiction, assume that there is an infinite subset $J$ of $\mathbb{N}$ such that for each $k\in J$, there exists an isomorphism $g_k : T \to S_k$. Since $S_k$ is a proper subtree of $T$, $g_k$ is a non-elliptic self-embedding of $T$ by Proposition \ref{HAL}. So, it preserves a maximal (double) ray $R_{g_k}$ with $g(R_{g_k})\subseteq R_{g_k}$. By the argument above, $g_k^+ = f^+$. If $R_f = R_{g_k}$, then $g_k$ is parabolic meaning that $g_k(r_0)\neq r_0$. But this is not possible because $r_0\in S_k$ is not covered by $g$. Therefore, for  $k\in J$, $R_f \neq R_{g_k}$ and $f^+ = g_k^+$. Let $k\in J$ be given and $r_{g_k}$ be the starting vertex of the ray $R_f\cap R_{g_k}$. We have $r_{g_k} \leq_f r_t$. Since there are only finitely many vertices $r$ with $r_1 \leq_f r \leq_f r_t$, for some $r_1 \leq_f u \leq_f r_t$ and some infinite set $I\subseteq J$,  $u$ is the starting vertex of the ray $R_f\cap R_{g_k}$ for each $k\in I$. 

\noindent\textbf{Case 1} $s \leq_f u$. Note that $s\in S_k$ for each $k\in I$. Therefore, $g_k^{-1}(s)$ is defined and we have $g_k^{-1}(s)\in T^f_u$ for each $k\in I$. Further, if $g_k^{-1}(u)=g_l^{-1}(u)$ for $k, l\in I$ with $k < l$, then $(g_l\circ g_k^{-1})(u)=u$ meaning that $g_l\circ g_k^{-1}$ fixes a vertex. We have $g_l\circ g_k^{-1}: S_k\to S_l$ and since $S_l$ is a proper subtree of $S_k$, it follows that $g_l\circ g_k^{-1}$ is a non-elliptic self-embedding of $S_k$ by Proposition \ref{HAL}, a contradiction. Hence, for distinct $k, l\in I$ we have $g_k^{-1}(s)\neq g_l^{-1}(s)$ meaning that $T^f_u$ has infinitely many vertices of degree at least 3, that is $T^f_u$ contains a rake, a contradiction. 

\noindent\textbf{Case 2} $u <_f s$. Since $s$ is the minimal vertex of $R_f$ w.r.t $\leq_f$ whose degree is greater than 2,  $deg(u)=2$ and thus $u=r_1$. In this case, for each $k\in I$, we have $g_k^{-1}(s)\in T^f_{r_1}$. Further,  $deg(g_k^{-1}(s))\geq 3$ for each $k\in I$. Similar to case 1 we observe that for two distinct $k, l\in I$, $g_k^{-1}(u)\neq g_l^{-1}(u)$. It follows that there are infinitely many vertices in $T^f_{r_1}$ with degree at least 3, that is $T^f_{r_1}$ contains a rake, a contradiction.

Therefore, $T$ is isomorphic to only finitely many siblings $S_k$, that is, there is a positive integer $k_1$ such that $T \ncong S_k$ for every $k \geq k_1$. In particular, $T \ncong S_{k_1}$. For $k_1$ we know that $f^{k_1+1}_{\upharpoonright S_{k_1}}$ is a parabolic self-embedding of $S_{k_1}$. By a similar argument above, there exists $k_1 < k_2$ such that $S_{k_1} \ncong S_{k_2}$ and clearly $T \ncong S_{k_2}$. Continuing this, we obtain the infinite sequence $T, S_{k_1}, S_{k_2}, \ldots$ of siblings of $T$ which are pairwise non-isomorphic. Hence, $Sib(T)=\infty$.
\end{proof}

Note that the proof of Lemma \ref{Parabolicregular} is independent of the number of directions. Further, $f$ in Lemma \ref{Parabolicregular} is obviously non-surjective. We show by the next proposition that if there exists a non-regular end preserved forward or backward (backward occurs when the end is preserved by a hyperbolic self-embedding), then the sibling number of $T$ is infinite. 

\begin{proposition} \label{Nonregularray}
Let T be a locally finite tree. If a non-regular ray in $T$ is preserved forward or backward by a non-elliptic self-embedding, then $Sib(T)=\infty$. 
\end{proposition} 

\begin{proof}
Let $R=r_1r_2\cdots$ be a non-regular ray preserved forward by a non-elliptic self-embedding $f$ and assume that $f(r_1)=r_m$, where $m>1$. For  $1\leq i\leq m-1$ set $B_i:=\{T^f_{f^k(r_i)} : k\in\mathbb{N}\}$. Since $R$ is non-regular, there is some $1\leq j\leq m-1$ such that $B_j$ contains infinitely many pairwise non-equimorphic trees. By Lemma \ref{MutualrootedIso} these trees are pairwise non-isomorphic. In other words, there are infinitely many $k$ such that the trees $T^f_{f^k(r_j)}$ are pairwise non-isomorphic. 
Let $(n_k)_{k<\omega}$ be an infinite subsequence of $(f^n(r_j))_{n<\omega}$ which is strictly increasing w.r.t $\leq_f$ such that $T^f_{n_k}\ncong T^f_{n_{k+1}}$. For each $k$ let $s_k, t_k\in R$ be such that $n_k \leq_f s_k <_f t_k \leq_f n_{k+1}$, $f(s_k)=t_k$ and $T^f_{s_k}\ncong T^f_{t_k}$. Since $T^f_{s_k}$ embeds into $T^f_{t_k}$ properly, it follows that $F_k:= T^f_{t_k}\setminus f(T^f_{s_k})$ is non-empty. Note that $R$ separates the $F_k$. In addition, $T\setminus f(T)$ contains the forest $F:=\bigcup_{k<\omega} F_k$ which has infinitely many components. Consequently, $Sib(T)=\infty$  by Theorem \ref{Infinitecomponents}. 

If a non-regular ray $R=r_1r_2\cdots$ is preserved backward by some non-elliptic self-embedding $f$, then $f$ is hyperbolic, $R\subset Z_f$ and $f(r_m)=r_1$ for some $m>1$. By a similar argument given above, there is an infinite sequence $(n_k)_{k<\omega}$ of vertices of $R$ which is strictly decreasing w.r.t $\leq_f$ and $T^f_{n_{k+1}}$ embeds in $T^f_{n_k}$ properly by some power of $f$. Then, by selecting $t_{n_{k+1}}\leq_f t_k <_f s_k \leq_f t_{n_k}$ for which $T^f_{t_k}\ncong T^f_{s_k}$ and $f(t_k)=s_k$, we have $F_k:=T^f_{s_k}\setminus f(T^f_{t_k})\neq\emptyset$. It follows that $T\setminus f(T)$ contains the forest $F:=\bigcup_{k<\omega}F_k$ with infinitely many components. Consequently, $Sib(T)=\infty$ by Theorem \ref{Infinitecomponents}. 
\end{proof}  

The proof of  Proposition \ref{Nonregularray} is also independent of the number of directions. We also note that since $T\setminus f(T)$ has infinitely many components, $f$ is a non-surjective self-embedding of $T$. 

\begin{corollary} \label{Nonregularend}
If for some non-elliptic self-embedding $f$ of a locally finite tree T, $f^+$ or $f^-$ is non-regular, then $Sib(T)=\infty$.
\end{corollary}

Now Lemma \ref{Parabolicregular} and Corollary \ref{Nonregularend} immediately imply the following.  

\begin{theorem} \label{ParabolicInfinite}
Let $T$ be a locally finite tree. If T has a parabolic self-embedding, then $Sib(T)=\infty$, unless $T$ is a ray.
\end{theorem} 

\begin{corollary}
A locally finite tree which does not have hyperbolic self-embedding has both TAP and ASP. 
\end{corollary}

\begin{proof}
Let $T$ be a locally finite tree whose monoid of self-embeddings does not have hyperbolic self-embedding. If all self-embeddings of $T$ are elliptic, then by Proposition \ref{Ellipticauto} they are automorphisms and $Sib(T)=1$. Note that automorphisms are surjective. If $T$ has a parabolic self-embedding, then $Sib(T)=\infty$ by Theorem \ref{ParabolicInfinite}, unless $T$ is a ray. Moreover, parabolic self-embeddings are non-surjective. Thus, in this case $T$ has both TAP and ASP. 
\end{proof}
 
\subsection{Sibling Number by Means of Directions} 

We collect all the results obtained in the following theorem to count the sibling number of a locally finite tree based on the number of directions.  Recall that by Theorem \ref{SizeofDm} a tree has either none, one, two or infinitely many directions.  

\begin{theorem}
Let T be a locally finite tree and M the monoid of self-embeddings of T.  
\begin{enumerate}
\item If $|\mathcal{D}(M)|=0$ or $2$, then $Sib(T)=1$. 
    \item If $|\mathcal{D}(M)|=1$, then $Sib(T)=1$ or $\infty$. 
    \item If $|\mathcal{D}(M)|=\infty$ and there is a parabolic self-embedding of T or there is a non-regular end preserved forward or backward by some non-elliptic self-embedding of T, then $Sib(T)=\infty$. 
\end{enumerate}
\end{theorem}

\begin{proof}
(1) If $|\mathcal{D}(M)|=0$ or 2, then $T$ has one sibling by Propositions \ref{Ellipticauto} and \ref{Twodirectionssibone}. 

(2) Assume that $|\mathcal{D}(M)|=1$. In this case if there is a parabolic self-embedding of $T$, then by Theorem \ref{ParabolicInfinite}, $Sib(T)=\infty$, unless $T$ is a ray. However,  when there is no parabolic self-embedding, then the unique direction $\eta$ is determined by some hyperbolic self-embedding $g\in M$ i.e. $\eta=g^+$. If $g^+$ or $g^-$ is non-regular, then $Sib(T)=\infty$ by Corollary \ref{Nonregularend} and if both $g^+$ and $g^-$ are regular, then $Sib(T)=1$ or $\infty$ by Theorem \ref{RegularStable} because $T$ is stable in this case. 

(3) Finally, assume that $|\mathcal{D}(M)|=\infty$. In this case, clearly $T$ is not a ray. If there is a parabolic self-embedding of $T$, then $Sib(T)=\infty$ by Theorem \ref{ParabolicInfinite} and if there is a non-regular end preserved forward or backward by some non-elliptic self-embedding, then $Sib(T)=\infty$ by Corollary \ref{Nonregularend}. 
\end{proof}

\section{Open Case}

Abdi, Laflamme, Tateno and Woodrow \cite{ALTW} provided a detailed explanation of locally finite trees having an arbitrary finite number of siblings. 
They showed that for each non-zero integer $n$ there are locally finite trees $T_1, \ldots, T_n$ as a finite set of pairwise non-isomorphic siblings such that each tree as a sibling of $T=T_1$ is isomorphic to some $T_k$, $1\leq k\leq n$. The case $n=2$ immediately provides an example not having TAP. Also, it is shown that the tree $T$ has non-surjective self-embeddings, meaning that it does not have ASP. 
Considering all the results in Section \ref{LFTrees} and the tree examples in \cite{ALTW} we may ask the following.

\begin{problem}
Let T be a locally finite tree and M the monoid of self-embeddings of T. Assume that 
\begin{enumerate}
   \item $|\mathcal{D}(M)|=\infty$,
  \item all non-elliptic self-embeddings of T are hyperbolic,
   \item there is no non-regular end preserved forward or backward by some self-embedding of T,
   \item for each self-embedding f of T, $T\setminus f(T)$ is a finite union of nearly finite trees.
\end{enumerate}
What are further conditions needed for T to have TAP or ASP?
\end{problem}

\vspace{0.5cm}
\noindent{\bf \large Acknowledgements}

I would like to thank my PhD supervisors Professor Robert Woodrow and Professor Claude Laflamme for suggesting this problem and for their help and advice.


\vspace{1cm}

\textsc{Department of Mathematics and Statistics, University of Calgary, Calgary, Alberta, Canada, T2N 1N4}

{\em Email address:} 
\url{davoud.abdikalow@ucalgary.ca}


\begin{thebibliography}{5}


\bibitem{ALTW}
D. Abdi, C. Laflamme, A. Tateno, R. Woodrow, {\em An example of Tateno disproving conjectures of Bonato-Tardif, Thomass\'e, and Tyomkyn}, (2022), \url{http://arxiv.org/abs/2205.14679}.

\bibitem{BT}
A. Bonato, C. Tardif, \textit{Mutually embeddable graphs and the tree alternative conjecture}, J. Combin. Theory Ser. B 96 (2006), 874-880. 

\bibitem{D}
R. Diestel. \textit{Graph Theory}. Fifth ed. Springer. (2017).

\bibitem{HA}
R. Halin, {\em Fixed configurations in graphs with small number of disjoint rays}, in: R. Bodendiek (Ed.), Contemporary
Methods in Graph Theory, Bibliographisches Inst., Mannheim, (1990), 639–649.

\bibitem{HAL}
R. Halin. \textit{Automorphisms and endomorphisms of infinite locally finite graphs}. Abh. Math. Sem. Univ. Hamburg 39 (1973), 251-283. 


\bibitem{HAM}
M. Hamann. \textit{Self-self-embeddings of trees}. Discrete Math. 342 (2019), no 12. 111586, 1-7.

\bibitem{HAMP}
M. Hamann, personal communication, November 2019. 

\bibitem{LPS}
C. Laflamme, M. Pouzet, N. Sauer. \textit{Invariant subsets of scattered trees and the tree alternative property of Bonato and Tardif}. Abh. Math. Semin. Univ. Hambg. 87 (2017), 369-408. 

\bibitem{TAT}
A. Tateno, \textit{Mutually embeddable trees and a counterexample to the Tree Alternative Conjecture}, unpublished manuscript, (2008), 32 pages. 


\bibitem{TY}
M. Tyomkyn. \textit{A proof of the rooted tree alternative conjecture}. Discrete Math 309 (2009), 5963-5967. 

\bibitem{TYP}
M. Tyomkyn. \textit{personal communication}, July 2020.

\end{thebibliography}
\end{document}